\documentclass[10pt,a4paper,oneside,final]{amsart}
\usepackage[T1]{fontenc}
\usepackage[utf8]{inputenc}
\usepackage{lmodern}
\usepackage[final]{microtype}
\usepackage[backend=biber,giveninits=true]{biblatex}
\usepackage{amssymb,mathtools,amsthm,tikz-cd}
\usepackage{hyperref,xurl}
\hypersetup{breaklinks=true,hidelinks,
  pdftitle={Uniqueness of differential Poincare duality models},
  pdfauthor={Pavel H\'ajek}}

\theoremstyle{plain}
\newtheorem{theorem}{Theorem}[section]
\newtheorem{proposition}[theorem]{Proposition}
\newtheorem{lemma}[theorem]{Lemma}
\newtheorem{corollary}[theorem]{Corollary}
\theoremstyle{remark}
\newtheorem{remark}[theorem]{Remark}
\newtheorem{example}[theorem]{Example}

\newcommand{\wh}[1]{\hat{#1}}
\newcommand{\wt}[1]{\tilde{#1}}
\newcommand{\dprime}{{\prime\prime}}
\newcommand{\N}{\mathbb N}
\newcommand{\K}{\mathbb K}
\newcommand{\Dd}{\mathrm d}
\newcommand{\Hh}{\mathrm H}
\newcommand{\HH}{\mathcal H}
\newcommand{\QQ}{\mathcal Q}
\newcommand{\Or}{\varepsilon}

\newcommand{\la}{\langle}
\newcommand{\ra}{\rangle}
\DeclareMathOperator{\im}{im}
\DeclareMathOperator{\Span}{span}
\newcommand{\Id}{{{\mathchoice{\rm 1\mskip-4mu l}{\rm 1\mskip-4mu l}{\rm 1\mskip-4.5mu l}{\rm 1\mskip-5mu l}}}}
\newcommand{\into}{\hookrightarrow}
\newcommand{\onto}{\twoheadrightarrow}

\begin{document}
\title[Uniqueness of differential Poincar\'e duality models]{Uniqueness of differential
Poincar\'e duality models}
\author{Pavel H\'ajek}
\address{University of Hamburg, Bundesstr.~55, 20146 Hamburg, Germany}
\date{}
\begin{abstract}
We prove the remaining even-dimensional case of the Lambrechts--Stanley
``uniqueness'' conjecture for 1-connected differential Poincar\'e duality
models. Together with our previous
odd-dimensional result, this shows that any two such algebras weakly
homotopy equivalent as PDGAs admit direct PDGA quasi-isomorphisms into a common
connected differential Poincar\'e duality algebra. We use our previous
method of obtaining differential Poincar\'e duality models as
nondegenerate quotients of Hodge extensions, the new ingredient being a Hodge extension in the middle degree.
\end{abstract}
\maketitle
\tableofcontents

\section{Introduction}\label{Sec:Introduction}

The uniqueness problem for differential Poincar\'e duality models asks
whether two weakly equivalent models admit quasi-isomorphisms into a
common connected model. Lambrechts \& Stanley \cite{Lambrechts2007} proved this under
additional connectivity assumptions and conjectured that these can be
removed~\cite[Theorem~7.1 and the subsequent conjecture]{Lambrechts2007}.
In~\cite[Propositions~5.5 and~5.6]{HodgePaper} we proved uniqueness for
1-connected algebras assuming either $\Hh^2(V)=0$ or that $n$ is odd.
Here we establish the remaining even-dimensional case, with no
assumption on $\Hh^2(V)$.

An \emph{oriented Poincar\'e DGA} (\emph{oriented PDGA})
$(V,\Dd,\Or)$ of degree $n\in\N_0$ over a field~$\K$ is a
non-negatively graded unital commutative differential graded algebra
(CDGA) $(V,\Dd)$ with a linear map $\Or\colon V\to\K$ satisfying
$\Or\circ\Dd=0$,
$\Or|_{V^i}=0$ for $i\ne n$, and $\Or|_{\ker\Dd}\ne0$, such that
the pairing
\[
                    \la v_1,v_2\ra\coloneqq\Or(v_1v_2)
\]
induces a perfect pairing on cohomology~$\Hh(V)$.
It is a \emph{differential Poincar\'e duality algebra}
(\emph{dPD algebra}) if this pairing is perfect on chain level as well.
A \emph{PDGA} specifies only the induced cohomology orientation
$\Or_*([v])=\Or(v)$; an oriented PDGA also specifies its chain-level
lift~$\Or$.
Such an algebra is \emph{connected} if $V^0=\K$ and
\emph{1-connected} if additionally $V^1=0$. It is of \emph{finite type}
if $\dim_\K V^k<\infty$ for every~$k$.

A PDGA morphism is a DGA morphism preserving
cohomology orientation. Two PDGAs are \emph{weakly homotopy equivalent}
if they can be connected by a zig-zag of PDGA quasi-isomorphisms.
For the uniqueness theorem and its corollary, assume
$\operatorname{char}\K=0$. We discuss the validity of the method in
positive characteristic other than~$2$ in Remark~\ref{Rem:Characteristic}.

\begin{theorem}\label{Thm:Uniqueness}
Let $(V,\Dd,\Or)$ and $(V^\prime,\Dd^\prime,\Or^\prime)$ be
1-connected dPD algebras of degree $n\in\N_0$ that are weakly homotopy
equivalent as PDGAs. Then there is a connected dPD algebra
$(V^\dprime,\Dd^\dprime,\Or^\dprime)$ and PDGA quasi-isomorphisms
\[
\begin{tikzcd}[row sep=small,column sep=small]
& V^\dprime &\\
V\arrow[hook]{ur}{\iota} && V^\prime\arrow[hook',swap]{ul}{\iota^\prime},
\end{tikzcd}
\]
which are injective and preserve chain-level orientation.
\end{theorem}

\begin{corollary}\label{Cor:Unoriented}
Suppose that the underlying DGAs of two 1-connected dPD algebras are
connected by a zig-zag of DGA quasi-isomorphisms. Then both admit
injective DGA quasi-isomorphisms into a common connected dPD algebra.
\end{corollary}

\begin{proof}
Let $f_*\colon\Hh(V)\to\Hh(V^\prime)$ be the isomorphism induced by the
zig-zag of DGA quasi-isomorphisms. Replace $\Or$ with the nonzero scalar
multiple whose induced
orientation is $\Or_*^\prime\circ f_*$. Orient the intermediate
cohomology algebras along the zig-zag and apply Theorem~\ref{Thm:Uniqueness}.
\end{proof}

We proved the result for odd~$n$ in~\cite[Theorem~1.5 and
Proposition~5.6]{HodgePaper}. To construct the common dPD target,
we first form the base-changed relative minimal Sullivan model of
multiplication $D=V\otimes V^\prime\otimes\Lambda(sU)$, where
$\Lambda U$ is a common minimal Sullivan model of finite type of $V$
and $V^\prime$, as in Lambrechts and Stanley~\cite{Lambrechts2007}.
Then $D$ is a connected oriented PDGA of finite type receiving
quasi-isomorphisms from both algebras.
We then extend $D$ to an
algebra of Hodge type and take its nondegenerate quotient. Although
the original algebras are 1-connected, $D$ may contain degree-one
elements. For odd~$n$, our connected extension theorem applies
directly; for even $n=2m$, it first requires a Hodge decomposition
in degree~$m$.

The new ingredient is Proposition~\ref{Prop:Extension}, which supplies
this middle-degree extension. For a connected oriented PDGA $V$ of
degree $n=2m$ of finite type, it assumes a pre-Hodge
decomposition $V=\HH\oplus\im\Dd\oplus C$ satisfying
\[
                         V^1\HH^m\subset C^{m+1}
\]
and constructs a connected extension that is Hodge in the middle degree
by adjoining exact partners dual to~$C^m$. In the proof of the theorem we
verify this condition for~$D$.
Remark~\ref{Rem:ConnectedExtension} explains how the same construction
also simplifies the proof of our earlier connected extension theorem for odd $n$.

Our counterexamples delimit both the theorem and this construction.

Example~\ref{Ex:ConnectedUniqueness} gives, for every $n\ge3$, weakly
homotopy equivalent connected dPD algebras with no common connected dPD
target; hence 1-connectedness cannot be dropped. This is analogous
to~\cite[Example~6.7]{HodgePaper}, which demonstrates that 2-connectedness
cannot be dropped if one requires the target to be 1-connected.

When $-1$ is not a square in~$\K$, Example~\ref{Ex:Connected}
constructs, in every positive even degree, a connected oriented PDGA
of finite type with no connected finite-type quasi-isomorphic
extension of Hodge type.

Finally, Example~\ref{Ex:MinimalSullivan} constructs for $n=58$ a
4-connected minimal Sullivan PDGA that is not of Hodge type: it admits
no Hodge decomposition for any lift of its cohomology orientation.
Over $\mathbb Q$, it is realized by a simply connected closed smooth
manifold. Thus passing to a minimal Sullivan model does not by itself
provide a Hodge decomposition and hence an induced dPD model via the
nondegenerate quotient.

\medskip
\noindent\textit{Acknowledgments:}
This paper benefited substantially from the assistance of OpenAI's GPT-6 Astra. It would not have been possible without the support of my family, Inessa and Markus H\'ajek, who gave me the time to complete it.
Some underlying ideas and drafts were developed during my postdoctoral appointment at the University of Hamburg. The manuscript was put together during my affiliation with the Wolfram Institute.

\section{Hodge extension in the middle degree}\label{Sec:Extension}

Recall from~\cite[Definition~4.1]{HodgePaper} that a \emph{pre-Hodge
decomposition} of an oriented PDGA is a decomposition
\[
             V=\HH\oplus\im\Dd\oplus C,\qquad C\perp\HH,
\]
where $\HH$ is a complement of $\im\Dd$ in $\ker\Dd$ and $C$ is a
complement of $\ker\Dd$ in~$V$. Such decompositions always exist by
\cite[Lemma~4.3]{HodgePaper}.
It is a \emph{Hodge decomposition} if, in addition, $C\perp C$; see
\cite[Definition~3.6]{HodgePaper}. We say that an oriented PDGA is
\emph{of Hodge type} if it admits a Hodge decomposition.
An extension $V\into\wh V$ of oriented PDGAs is an injective DGA
morphism that preserves the chain-level orientation. A \emph{PDGA
retraction} of this extension is a PDGA morphism
$\pi\colon\wh V\onto V$ such that
$\pi|_V=\Id_V$; see~\cite[Definition~4.10]{HodgePaper}.
Every PDGA retraction is a quasi-isomorphism: its map on cohomology
is surjective because it is a retraction, and injective because it
preserves the perfect pairing; see~\cite[Remark~4.11]{HodgePaper}.

For connected oriented PDGAs of finite type, our connected extension
theorem~\cite[Proposition~4.13]{HodgePaper} requires a Hodge decomposition
in the middle degree when the degree is positive and even.
The following proposition supplies the missing middle-degree extension
under an adaptedness condition.

\begin{proposition}\label{Prop:Extension}
Let $\K$ be a field of characteristic different from~$2$, and let
$(V,\Dd,\Or)$ be a connected oriented PDGA of degree $n=2m\ge4$
of finite type over~$\K$. Suppose that $V$ admits a pre-Hodge
decomposition $V=\HH\oplus\im\Dd\oplus C$ such that
\begin{equation}\label{Eq:Adapted}
                         V^1\HH^m\subset C^{m+1},
\end{equation}
where $V^1\HH^m\coloneqq\Span_\K\{vh\mid v\in V^1,\ h\in\HH^m\}$.
Then $V$ admits a connected extension of Hodge type of finite type
that retracts onto~$V$ as a PDGA.
\end{proposition}

\begin{proof}
By~\cite[Proposition~4.13]{HodgePaper}, it suffices to construct a
connected extension $V\into\wh V$ of finite type that retracts
onto~$V$ by a PDGA morphism $\pi$ and admits a Hodge decomposition in
degree~$m$. Indeed, applying that proposition to $\wh V$ then gives a
connected extension of Hodge type, and the extensions and retractions compose.
We adjoin exact elements dual to $C^m$ and use them to make the
middle-degree coexact space perpendicular to itself.
We first treat $m\ge3$, assuming that either
$\operatorname{char}\K=0$ or $m$ is even.

Let $V=\HH\oplus\im\Dd\oplus C$ be the pre-Hodge decomposition from
the assumption, so that $V^1\HH^m\subset C^{m+1}$.
Choose a basis $c_\alpha$ ($\alpha\in I$) of $C^m$ and let
$\phi_\alpha\colon C\to\K$ be the coordinate functionals, extended
by zero outside degree~$m$. Consider the free CDGA
\[
 \Lambda\coloneqq\Lambda_\K(x_\alpha,\Dd x_\alpha\mid\alpha\in I),
 \qquad \deg x_\alpha=m-1,
\]
and set
\[
 (\wh V,\wh\Dd)\coloneqq(\Lambda,\Dd)\otimes(V,\Dd).
\]
Since $\Hh(\Lambda)=\K$ by~\cite[Lemma~4.7]{HodgePaper},
K\"unneth's theorem gives quasi-isomorphisms
\[
                         V\into\wh V\xrightarrow{\pi}V,
\]
where $\pi$ is induced by the augmentation of~$\Lambda$.
Moreover, $\wh V$ is connected and of finite type because $I$ is
finite and the new generators have positive degrees.

To equip $\wh V$ with an orientation extending~$\Or$, set
$\wh\Or|_V=\Or$ and, for $v=h+\Dd c+c^\prime$ with $h\in\HH$
and $c,c^\prime\in C$, define
\begin{equation}\label{Eq:Orientation}
 \wh\Or(x_\alpha v)\coloneqq(-1)^m\phi_\alpha(c),
 \qquad
 \wh\Or(\Dd x_\alpha v)\coloneqq\phi_\alpha(c^\prime).
\end{equation}
Set $\wh\Or=0$ on monomials containing at least two new generators.
This defines a linear map of degree~$-n$. Since $\Dd v=\Dd c^\prime$,
we have
\[
 \begin{aligned}
  \wh\Or\bigl(\wh\Dd(x_\alpha v)\bigr)
   &=\phi_\alpha(c^\prime)
      +(-1)^{m-1}(-1)^m\phi_\alpha(c^\prime)=0,\\
  \wh\Or\bigl(\wh\Dd(\Dd x_\alpha v)\bigr)
   &=(-1)^m\wh\Or(\Dd x_\alpha\Dd c^\prime)=0.
 \end{aligned}
\]
The differential preserves the number of new generators in a
monomial, so $\wh\Or\circ\wh\Dd=0$ on the remaining monomials as well.
Thus $\wh\Or$ is an orientation extending~$\Or$, and $\pi$ is a
PDGA retraction.

It remains to construct a Hodge decomposition of $\wh V$ in
degree~$m$. Set
\[
 X\coloneqq\Span_\K\{x_\alpha\mid\alpha\in I\},
 \qquad Z\coloneqq X\cdot V^1.
\]
Since $2(m-1)>m$, we have
\[
 \wh V^{m-1}=V^{m-1}\oplus X,\qquad
 \wh V^m=V^m\oplus\Dd X\oplus Z.
\]
We define $\wh C^m\coloneqq C^m\oplus Z$. Then
\begin{equation}\label{Eq:Middle}
 \wh V^m
   =\HH^m\oplus
       \underbrace{\bigl((\im\Dd)^m\oplus\Dd X\bigr)}_{(\im\wh\Dd)^m}
       \oplus\underbrace{\bigl(C^m\oplus Z\bigr)}_{\wh C^m}.
\end{equation}
To see that $\wh C^m$ is a complement of $(\ker\wh\Dd)^m$, consider
$c\in C^m$ and $v_\alpha\in V^1$,
\[
 \wh\Dd\left(c+\sum_\alpha x_\alpha v_\alpha\right)
 =\Dd c+\sum_\alpha\Dd x_\alpha v_\alpha
       +(-1)^{m-1}\sum_\alpha x_\alpha\Dd v_\alpha.
\]
If this differential vanishes, its component in $\Dd X\cdot V^1$
gives $v_\alpha=0$ for every~$\alpha$, and then $\Dd c=0$ gives
$c=0$. Since the other direct summands in~\eqref{Eq:Middle} are
contained in $(\ker\wh\Dd)^m$, it follows that $\wh C^m$ is a
complement of $(\ker\wh\Dd)^m$.
To see that $\wh C^m \perp\HH^m$, let
$c\in C^m$,
$v_\alpha\in V^1$, and $h\in\HH^m$. We have
\[
 \left\la c+\sum_\alpha x_\alpha v_\alpha,h\right\ra_{\wh\Or}
 =\la c,h\ra_\Or
      +\sum_\alpha\wh\Or\bigl(x_\alpha(v_\alpha h)\bigr)=0,
\]
because $C\perp\HH$ and $v_\alpha h\in C^{m+1}$ by~\eqref{Eq:Adapted}.
Thus the adaptedness condition ensures that the additional coexact
elements remain perpendicular to~$\HH^m$.
In the other degrees, choose complements
of $\im\wh\Dd$ in $\HH^\perp$ and denote their direct sum with
$\wh C^m$ by~$\wh C$.
We obtain a pre-Hodge decomposition $\wh V=\HH\oplus\im\wh\Dd\oplus\wh C$.

Write $\la-,-\ra$ for the pairing induced by~$\wh\Or$.
We seek a linear map $\mu\colon\wh C^m\to(\im\wh\Dd)^m$ satisfying
\[
 \la\mu(u),u^\prime\ra+\la u,\mu(u^\prime)\ra+\la u,u^\prime\ra=0
 \qquad(u,u^\prime\in\wh C^m).
\]
This is the middle-degree \emph{Hodge-twist condition}
of~\cite[Definition~4.4]{HodgePaper}. It ensures that, after extending
$\mu$ by zero outside degree~$m$, the space
\[
 \wh C_\mu\coloneqq\operatorname{graph}(\mu)
   =\left(\bigoplus_{i\ne m}\wh C^i\right)
      \oplus\{u+\mu(u)\mid u\in\wh C^m\}
\]
gives a pre-Hodge decomposition
$\wh V=\HH\oplus\im\wh\Dd\oplus\wh C_\mu$
that is Hodge in degree~$m$.
Indeed, as in the proof of~\cite[Lemma~4.5]{HodgePaper}, adding exact
elements preserves the complement to $\ker\wh\Dd$ and its
orthogonality to~$\HH$. Moreover, exact elements are perpendicular
to closed elements, so $\la\mu(u),\mu(u^\prime)\ra=0$.
The displayed condition therefore says precisely that
$\wh C_\mu^m\perp\wh C_\mu^m$.

To construct $\mu$, recall that the definition of $\wh\Or$ gives
\[
 \la\Dd x_\alpha,c_\beta\ra=\delta_{\alpha\beta},
 \qquad
 Z\oplus\Dd X\perp Z\oplus\Dd X.
\]
Define $\mu\colon\wh C^m\to\Dd X$ by
\begin{equation}\label{Eq:Correction}
 \begin{aligned}
 \mu(c+z)&\coloneqq-\sum_{\alpha\in I}
       \left\la\frac12c+z,c_\alpha\right\ra\Dd x_\alpha\\
 &=-\sum_{\alpha\in I}
       \left(\frac12\la c,c_\alpha\ra+\la z,c_\alpha\ra\right)
                 \Dd x_\alpha
 \qquad(c\in C^m,\ z\in Z).
 \end{aligned}
\end{equation}
It remains to check the middle-degree Hodge-twist condition.
For $u=c+z$ and $u^\prime=c^\prime+z^\prime$ in $\wh C^m$, write
\[
 c=\sum_\alpha a_\alpha c_\alpha,\qquad
 c^\prime=\sum_\alpha a^\prime_\alpha c_\alpha.
\]
Using $\la\Dd x_\alpha,c_\beta\ra=\delta_{\alpha\beta}$ and
$\Dd X\perp Z$, we obtain
\[
 \begin{aligned}
 \la\mu(u),u^\prime\ra
   &=-\sum_\alpha
       \left(\frac12\la c,c_\alpha\ra+\la z,c_\alpha\ra\right)
       a^\prime_\alpha\\
   &=-\frac12\la c,c^\prime\ra-\la z,c^\prime\ra,\\
 \la u,\mu(u^\prime)\ra
   &=-(-1)^m\sum_\alpha a_\alpha
       \left(\frac12\la c^\prime,c_\alpha\ra+\la z^\prime,c_\alpha\ra\right)\\
   &=-\frac12\la c,c^\prime\ra-\la c,z^\prime\ra,\\
 \la u,u^\prime\ra
   &=\la c,c^\prime\ra+\la c,z^\prime\ra+\la z,c^\prime\ra.
 \end{aligned}
\]
Here the second identity follows from graded symmetry and the first,
and the third uses $Z\perp Z$. Their sum is zero, so $\mu$ satisfies
the middle-degree Hodge-twist condition. Thus
$\wh V=\HH\oplus\im\wh\Dd\oplus\wh C_\mu$ is a pre-Hodge
decomposition that is Hodge in degree~$m$.

For $m=2$, the algebra $\Lambda$ is acyclic over every field of
characteristic different from~$2$, and the same orientation makes
$\wh V=\Lambda\otimes V$ a connected extension of finite type that
retracts onto~$V$. Compared with the preceding calculation, $\wh V^2$
has the extra summand
\[
 W\coloneqq\Span_\K\{x_\alpha x_\beta\mid\alpha<\beta\},
\]
where we have chosen an ordering of~$I$. Set
\[
 Z\coloneqq(X\cdot V^1)\oplus W,\qquad
 \wh C^2\coloneqq C^2\oplus Z.
\]
For $\alpha<\beta$,
\[
 \wh\Dd(x_\alpha x_\beta)=\Dd x_\alpha x_\beta-x_\alpha\Dd x_\beta,
\]
and these images are linearly independent. They contain two new
generators and hence cannot cancel $\Dd(C^2)$ or
$\wh\Dd(X\cdot V^1)$, so~\eqref{Eq:Middle} remains valid. Moreover,
$W\perp\wh V$ by the definition of $\wh\Or$. Thus the preceding
construction applies with
$\mu|_W=0$ and gives a Hodge decomposition in degree~$2$.

It remains to treat odd $m\ge3$ when $\operatorname{char}\K=p>2$.
Here $\deg x_\alpha=m-1$ is even, so $\Lambda$ is not in general
acyclic: the classes of $x_\alpha^p$ and
$x_\alpha^{p-1}\Dd x_\alpha$ survive by~
\cite[Lemma~4.7(a)]{HodgePaper}. We therefore replace $\Lambda$ by the
acyclic extension $\mathcal A$ from~\cite[Definition~4.6]{HodgePaper}.
For each $\alpha$, it has additional generators
$x_{\alpha,i},y_{\alpha,i}$ ($i\ge1$), where $x_{\alpha,0}=x_\alpha$ and
\[
 \begin{aligned}
  \deg x_{\alpha,i}&=p^i(m-1),
  &\Dd x_{\alpha,i}&=x_{\alpha,i-1}^{p-1}\Dd x_{\alpha,i-1},\\
  \deg y_{\alpha,i}&=p^i(m-1)-1,
  &\Dd y_{\alpha,i}&=x_{\alpha,i-1}^{p}.
 \end{aligned}
\]
By~\cite[Lemma~4.7(b)]{HodgePaper}, $\Hh(\mathcal A)=\K$. The added
generator degrees tend to infinity, so
$\wh V\coloneqq\mathcal A\otimes V$ is again connected and of finite
type, and its augmentation is a quasi-isomorphism onto~$V$.
Extend~\eqref{Eq:Orientation} by zero on monomials containing an
additional generator. This remains an orientation because their
differentials contain either another additional generator or at least
$p\ge3$ original generators. Finally, every additional generator has
degree at least $p(m-1)-1>m$, so the preceding construction is unchanged
in degrees $m-1$ and~$m$ and again gives a Hodge decomposition in
degree~$m$.
\end{proof}

\begin{remark}[Simplifying the connected extension theorem]
\label{Rem:ConnectedExtension}
The same method simplifies the proof of the connected extension
theorem~\cite[Proposition~4.13]{HodgePaper}, where~$V$ is assumed to
be of finite type. The main complication in that proof came from
adjoining acyclic pairs only for a basis of the nondegenerate quotient
$\QQ^k(C)$. The extension can create new classes in $\QQ^k(\wh C)$,
requiring further extensions at the same degree.
Here we avoid this complication by using the full coexact space:
for $k>n/2$, choose a finite basis $c_\alpha$ of
$C^{n-k}$ and adjoin acyclic pairs as above with
$\deg x_\alpha=k-1$ and
$\la\Dd x_\alpha,c_\beta\ra=\delta_{\alpha\beta}$. The correction
\[
 \mu(u)\coloneqq-\sum_\alpha\la u,c_\alpha\ra\Dd x_\alpha
 \qquad(u\in\wh C^k)
\]
gives
\[
 \la u+\mu(u),c_\beta\ra
 =\la u,c_\beta\ra
   -\sum_\alpha\la u,c_\alpha\ra\delta_{\alpha\beta}=0.
\]
This is the Hodge-twist condition in degree $k>n/2$
from~\cite[Definition~4.4(1)]{HodgePaper}.
The correction therefore replaces the repeated extensions at fixed $k$ and the
dimension-stabilization argument in that proof.

Thus, the main technical extension steps
in~\cite{HodgePaper,Lambrechts2007} reduce to adjunction of acyclic generators and explicit linear algebra.
\end{remark}

\begin{remark}[Obstruction in characteristic two]
\label{Rem:CharacteristicTwo}
In characteristic~$2$, even a cyclic cochain complex with a perfect
pairing need not admit a Hodge decomposition; see
\cite[Example~6.1]{HodgePaper}.  Moreover, in
\cite[Example~6.2]{HodgePaper} we constructed, in every even degree
$n=2m\ge6$, a 1-connected dPD algebra with $\Hh^m=0$ that admits no
quasi-isomorphic extension of Hodge type. Thus the assumption
$\operatorname{char}\K\ne2$ in Proposition~\ref{Prop:Extension} is
genuine even when~\eqref{Eq:Adapted} is vacuous.
For connected oriented PDGAs of finite type, the obstruction to a
Hodge extension is confined to the middle degree: once that degree
admits a Hodge decomposition, \cite[Proposition~4.13]{HodgePaper}
applies over any field.
\end{remark}

\section{Proof of uniqueness in even degree}\label{Sec:Uniqueness}

Let $V,V^\prime$ be as in Theorem~\ref{Thm:Uniqueness}, with
$n=2m\ge6$. Following the proof of~\cite[Proposition~5.5]{HodgePaper},
choose a common minimal Sullivan model $(\Lambda U,\Dd)$ of finite
type, with $U^0=U^1=0$, equipped with an orientation and PDGA
quasi-isomorphisms
\[
             V\xleftarrow{\rho}\Lambda U
                 \xrightarrow{\rho^\prime}V^\prime.
\]
To construct a common target for $V$ and $V^\prime$, consider the
relative minimal Sullivan model of multiplication
\[
 \phi\colon D^\prime=\Lambda U\otimes\Lambda U\otimes\Lambda(sU)
                 \longrightarrow\Lambda U,
\]
where $|su|=|u|-1$. Its base change
\begin{equation}\label{Eq:Relative}
 D\coloneqq(V\otimes V^\prime)
             \otimes_{\Lambda U\otimes\Lambda U}D^\prime
       =V\otimes V^\prime\otimes\Lambda(sU)
\end{equation}
along $\rho\otimes\rho^\prime$ fits in the diagram
\[
\begin{tikzcd}[column sep=large,row sep=large]
 \Lambda U\otimes\Lambda U
   \arrow[hook]{r} \arrow[swap]{d}{\rho\otimes\rho^\prime\,\simeq}
   \arrow[bend left=25]{rr}{\mathrm{mult.}}
 & D^\prime \arrow{r}{\phi\,\simeq}
   \arrow{d}{\rho\otimes\rho^\prime\otimes\Id\,\simeq}
 & \Lambda U\\
 V\otimes V^\prime \arrow[hook]{r} & D.
\end{tikzcd}
\]
The last equality in~\eqref{Eq:Relative} is an equality of graded
algebras; the differential on $D$ is induced from that on $D^\prime$.
In the diagram, the left vertical map is a quasi-isomorphism by
K\"unneth's theorem. The right vertical map is a quasi-isomorphism
because $D^\prime$ is semi-free over $\Lambda U\otimes\Lambda U$
and the square is a pushout. The map $\phi$ is a quasi-isomorphism
by construction of the relative minimal model.
As $\phi$ restricts to the identity on each copy of $\Lambda U$,
the inclusions $V,V^\prime\into D$ are also quasi-isomorphisms.
Since $V,V^\prime,U$ are of finite type and $sU$ is concentrated in
positive degrees, $D$ is connected and of finite type.
As in the proof of~\cite[Proposition~5.5]{HodgePaper}, equip $D$ with
a chain-level orientation $\Or_D$ for which both inclusions preserve
orientation.

To finish the construction, it remains to find a connected Hodge
extension $D\into\wh D$ of finite type. Its nondegenerate quotient
$V^\dprime=\QQ(\wh D)$ is then a connected dPD algebra, and the
quotient map $\pi_\QQ\colon\wh D\onto V^\dprime$ is an
orientation-preserving quasi-isomorphism
by~\cite[Lemma~3.15(a)]{HodgePaper}. The composites
\[
 V\into D\into\wh D\xrightarrow{\pi_\QQ}V^\dprime,
 \qquad
 V^\prime\into D\into\wh D\xrightarrow{\pi_\QQ}V^\dprime
\]
are therefore PDGA quasi-isomorphisms preserving chain-level
orientation and are injective by~\cite[Lemma~3.13]{HodgePaper}.
Proposition~\ref{Prop:Extension} supplies the required Hodge extension
provided that $D$ satisfies the following lemma.

\begin{lemma}\label{Lem:Adapted}
The oriented PDGA $(D,\Dd,\Or_D)$ admits a pre-Hodge decomposition
$D=\HH\oplus\im\Dd\oplus C$ such that
\[
                         D^1\HH^m\subset C^{m+1}.
\]
\end{lemma}

\begin{proof}
Choose a space of representatives
$\HH\subset V\otimes1\otimes1\subset D$ of
$\Hh(D)\simeq\Hh(V)$ and set
\[
                              Z\coloneqq D^1\HH^m.
\]
To choose $C^{m+1}$ containing $Z$, it suffices to show that
$Z\subset\HH^\perp$ and $Z\cap\ker\Dd_D=0$.
Indeed, since the pairing on $\HH$ is perfect, we have
$D=\HH\oplus\HH^\perp$ and
$\im\Dd_D=\HH^\perp\cap\ker\Dd_D$.
Choose a basis $u_\alpha$ of $U^2$.
Since $D^1=s(U^2)$, every element of $Z$ can be written as
$\sum_\alpha(su_\alpha)h_\alpha$ with $h_\alpha\in\HH^m$.
Let $\la-,-\ra_D$ denote the pairing induced by~$\Or_D$.
For $h^\prime\in\HH^{m-1}$, we have
\[
 \begin{aligned}
 \left\la\sum_\alpha(su_\alpha)h_\alpha,h^\prime\right\ra_D
 &=\sum_\alpha\la(su_\alpha)h_\alpha,h^\prime\ra_D\\
 &=\sum_\alpha\la su_\alpha,h_\alpha h^\prime\ra_D=0,
 \end{aligned}
\]
since $h_\alpha h^\prime\in V^{n-1}=0$ by 1-connectedness and
chain-level Poincar\'e duality on~$V$. All other degrees of $\HH$
are orthogonal to $Z$ for degree reasons, so $Z\subset\HH^\perp$.

To prove injectivity of $\Dd_D$ on $Z$, we use the explicit
differential on $D^\prime$ from~\cite[Section~2, Proposition~3]{Gatsinzi2016}:
\begin{equation}\label{Eq:MultiplicationDifferential}
 \begin{aligned}
 \Dd_{D^\prime}(v\otimes1\otimes1)
   &=\Dd v\otimes1\otimes1,\\
 \Dd_{D^\prime}(1\otimes v\otimes1)
   &=1\otimes\Dd v\otimes1,\\
 \Dd_{D^\prime}(1\otimes1\otimes su)
   &=u\otimes1\otimes1-1\otimes u\otimes1\\
   &\quad+\sum_{j\ge1}\frac{(-1)^j}{j!}
              (S\Dd_{D^\prime})^j(u\otimes1\otimes1),
 \end{aligned}
\end{equation}
where $v\in\Lambda U$, $u\in U$, and $S$ is the derivation of
degree $-1$ determined by
\[
 \begin{aligned}
 S(u\otimes1\otimes1)&=1\otimes1\otimes su,\\
 S(1\otimes u\otimes1)&=1\otimes1\otimes su,\\
 S(1\otimes1\otimes su)&=0.
 \end{aligned}
\]
For the degree-two generators
$u_\alpha$, minimality gives $\Dd u_\alpha=0$, so the sum vanishes.
The differential on these generators and its base change are therefore
\begin{equation}\label{Eq:ClosedGenerator}
 \begin{aligned}
 \Dd_{D^\prime}(1\otimes1\otimes su_\alpha)
   &=u_\alpha\otimes1\otimes1-1\otimes u_\alpha\otimes1,\\
 \Dd_D(1\otimes1\otimes su_\alpha)
   &=\rho(u_\alpha)\otimes1\otimes1
                   -1\otimes\rho^\prime(u_\alpha)\otimes1.
 \end{aligned}
\end{equation}
Since the representatives $h_\alpha$ are closed, it follows that
\[
 \begin{aligned}
 \Dd_D\Bigl(\sum_\alpha(su_\alpha)h_\alpha\Bigr)
 &=\sum_\alpha\Dd_D(su_\alpha)h_\alpha\\
 &=\sum_\alpha\bigl(\rho(u_\alpha)\otimes1\otimes1
             -1\otimes\rho^\prime(u_\alpha)\otimes1\bigr)
                  (h_\alpha\otimes1\otimes1)\\
 &=\sum_\alpha\rho(u_\alpha)h_\alpha\otimes1\otimes1
             -\sum_\alpha h_\alpha\otimes\rho^\prime(u_\alpha)\otimes1.
 \end{aligned}
\]
If this differential vanishes, its component in
$V^m\otimes V^{\prime2}\otimes1$ gives
$\sum_\alpha h_\alpha\otimes\rho^\prime(u_\alpha)=0$.
Since $U^2\simeq\Hh^2(\Lambda U)$ and $\rho^\prime$ is a
quasi-isomorphism, the elements $\rho^\prime(u_\alpha)$ are linearly
independent. Hence every $h_\alpha$ is zero, so $\Dd_D$ is injective on~$Z$.

We can thus extend $Z$ to a complement $C^{m+1}$ of
$(\im\Dd_D)^{m+1}$ in $(\HH^\perp)^{m+1}$. In every other degree $k$,
choose $C^k$ to be any complement of $(\im\Dd_D)^k$ in $(\HH^\perp)^k$.
This proves the lemma.
\end{proof}

This completes the proof of Theorem~\ref{Thm:Uniqueness} for $n=2m\ge6$.
For $n=0,2,4$, 1-connectedness and chain-level Poincar\'e duality give
\[
 \begin{array}{ll}
 n=0:&V=V^0=\K,\\
 n=2:&V=V^0\oplus V^2,\\
 n=4:&V=V^0\oplus V^2\oplus V^4,
       \quad V^3\simeq(V^1)^*=0.
 \end{array}
\]
Thus $\Dd=0$, and likewise $\Dd^\prime=0$. The zig-zag induces an
oriented algebra isomorphism $f\colon V\to V^\prime$, so we can take
$V^\dprime=V^\prime$, $\iota=f$, and $\iota^\prime=\Id$.
This completes the proof in all dimensions.

\begin{remark}[Characteristic zero]\label{Rem:Characteristic}
The assumption $\operatorname{char}\K=0$ in
Theorem~\ref{Thm:Uniqueness} is used for the Sullivan-model results.
As discussed in~\cite[Remark~5.8]{HodgePaper}, we lack references for
these results in positive characteristic. If suitable analogues were
available, the proof would work over every field of characteristic
different from~$2$. The remaining characteristic-two obstruction is the
middle-degree one explained in Remark~\ref{Rem:CharacteristicTwo}.
\end{remark}

\section{Counterexamples}\label{Sec:Examples}

\begin{example}[Connected dPD algebras with no common connected target]
\label{Ex:ConnectedUniqueness}
Over any field $\K$ and for every $n\ge3$, we construct weakly homotopy
equivalent connected dPD algebras $V,V^\prime$ with no common connected target.

For the first algebra, let
\[
V\coloneqq\Lambda_\K(a_1,b_1,\alpha_{n-1},\beta_{n-1})/I,
\qquad
\Dd=0,
\]
where $I$ is generated by $a^2$, $b^2$, $ab$, $a\beta$, $b\alpha$, $a\alpha-b\beta$,
and all elements of degree~$>n$. Setting $w\coloneqq a\alpha=b\beta$
and $\Or(w)=1$, we have
\[
V=\Span_\K\{1,a,b,\alpha,\beta,w\}.
\]
We have $a\alpha=b\beta=w$, so $V$ is a connected dPD algebra.

For the second algebra, let
\[
\begin{aligned}
V^\prime&\coloneqq\Lambda_\K(a_1,b_1,c_1,\eta_{n-2},\gamma_{n-1})/I^\prime,\\
\Dd c&=ab,\qquad \Dd\eta=\gamma,\qquad \Dd a=\Dd b=\Dd\gamma=0,
\end{aligned}
\]
where $I^\prime$ is the differential ideal generated by $a^2$, $b^2$,
$c^2$, $ac$, $bc$, $c\eta$, $\eta^2$, $a\gamma$,
$b\gamma$, $\eta\gamma$, and all elements of degree~$>n$. Setting
$\alpha\coloneqq b\eta$, $\beta\coloneqq-a\eta$,
$w\coloneqq ab\eta=c\gamma$, and $\Or(w)=1$, we have
\[
V^\prime=\Span_\K\{1,a,b,c,ab,\eta,\alpha,\beta,\gamma,w\},
\]
with $\Dd c=ab$, $\Dd\eta=\gamma$, and all other basis elements
closed. We have
\[
 a\alpha=w,\qquad b\beta=w,\qquad
 c\gamma=c\Dd\eta=(\Dd c)\eta-\Dd(c\eta)=ab\eta=w,
 \qquad (ab)\eta=w,
\]
so $V^\prime$ is a connected dPD algebra. 

To show that the algebras are weakly homotopy equivalent, consider
\[
E\coloneqq\Span_\K\{1,a,b,c,ab,\alpha,\beta,w\}\subset V^\prime.
\]
This is a sub-CDGA, and the projection $\pi\colon E\to V$ sending $c$ and $ab$ to zero and fixing the remaining basis elements is a DGA morphism. Its kernel is the acyclic pair $\Dd c=ab$, and the cokernel of $E\hookrightarrow V^\prime$ is the acyclic pair $\Dd\eta=\gamma$, so we have the zig-zag of DGA quasi-isomorphisms
\[
V\xleftarrow{\pi}E\hookrightarrow V^\prime.
\]
Setting $\Or_E(w)=1$ yields a zig-zag of PDGA quasi-isomorphisms.

It remains to rule out a common connected target. Suppose that there are DGA quasi-isomorphisms
\[
f\colon V\to T,
\qquad
g\colon V^\prime\to T,
\]
with $T$ connected. Since $(\im\Dd_T)^1=0$ and $f_*$ is an
isomorphism in degree~$1$, the elements $f(a),f(b)$ form a basis of
$(\ker\Dd_T)^1$. Their pairwise products vanish, since
$a^2=b^2=ab=0$ in $V$. Hence every product of closed degree-one
elements of $T$ vanishes. In particular,
\[
 g(w)=g(a)g(b)g(\eta)=0.
\]
This contradicts the fact that $g_*$ is an isomorphism in degree~$n$.
\end{example}

\begin{example}[No connected finite-type quasi-isomorphic extension of Hodge type]\label{Ex:Connected}
Let $\operatorname{char}\K\ne2$ and suppose that $-1$ is not a
square in $\K$. For every positive even degree $n=2m$, we construct a connected
oriented PDGA $V$ of finite type with no connected quasi-isomorphic
extension of Hodge type of finite type.

First let $n=2$. Set
\[
 \begin{aligned}
 V&\coloneqq\Lambda_\K(x_1,y_1,p_2,q_2)/I,\\
 \Dd x&=p,\qquad \Dd y=q,\qquad \Dd p=\Dd q=0,
 \end{aligned}
\]
where $I$ is generated by all elements of degree $>2$. Set
$\Or(xy)=1$ and $\Or(p)=\Or(q)=0$. Then
\[
 V=\Span_\K\{1,x,y,p,q,xy\},\qquad
 \Hh(V)=\Span_\K\{[1],[xy]\},
\]
so $V$ is a connected oriented PDGA.

Suppose that $V\into\wh V$ is a connected quasi-isomorphic extension of Hodge type. Since
$\Hh^1(\wh V)=0$ and
$(\im\wh\Dd)^1=0$, we have $(\ker\wh\Dd)^1=0$. A Hodge
decomposition therefore has $\wh C^1=\wh V^1$. In particular,
$x,y\in\wh C^1$, whereas $\wh C\perp\wh C$ would give
$\wh\Or(xy)=0$. This contradicts $\wh\Or(xy)=\Or(xy)=1$.

Now let $n=2m\ge4$ and suppose first that $m$ is even. Let
\[
 \begin{aligned}
 V&\coloneqq\Lambda_\K(a_1,b_{2m-1},h_m,c_m)/I,\\
 \Dd c&=ah,\qquad \Dd a=\Dd b=\Dd h=0,
 \end{aligned}
\]
where $I$ is generated by $ac$, $hc$, $h^2-ab$, $c^2+ab$, and all
elements of degree $>2m$. Setting $u\coloneqq ah=\Dd c$,
$w\coloneqq ab$, and $\Or(w)=1$, we have
\[
 V=\Span_\K\{1,a,b,h,c,u,w\},\qquad
 \Hh(V)=\Span_\K\{[1],[a],[b],[h],[w]\}.
\]
Since $[a][b]=[h]^2=[w]$, the orientation gives Poincar\'e duality
on cohomology.

Suppose that $\iota\colon V\into\wt V$ is a connected quasi-isomorphic
extension of Hodge type of finite type, and put
$P\coloneqq\QQ(\wt V)$; by~\cite[Lemma~3.15(a)]{HodgePaper}, the
quotient map is an orientation-preserving quasi-isomorphism, and $P$
is a dPD algebra. Keep the same names $a,b,h,c,w$ for the images under
the composite DGA map $V\to\wt V\to P$, and write $\Or_P$ for its
orientation; then $\Hh^m(P)=\K[h]$
with $\Or_P(h^2)=\Or_P(w)=1$.

Since $\Hh^m(P)=\K[h]$ with $\Or_P(h^2)=1$, and $-1$ is not a square
in $\K$, no closed element of $P^m$ can square to $-w$.
In $V$, the element $c$ does square to $-w$, but it is not closed,
since $\Dd c=ah\ne0$. We will show that its image in $P$ can be
deformed into a closed element of the form $z\coloneqq c+ae$,
where $e\in P^{m-1}$. This will give a contradiction, since
$z^2=c^2=-w$ by $a^2=ca=0$.
We have $\Dd z=ah-a\,\Dd e$, so such $e$ exists provided that
$ah\in a\,\Dd(P^{m-1})$.
For any $y\in P^{m-1}$ perpendicular to $a\,\Dd(P^{m-1})$, we have
$0=\Or_P(a\,\Dd x\,y)=\Or_P\bigl(x\,\Dd(ay)\bigr)$
for every $x\in P^{m-1}$, so $\Dd(ay)=0$ by perfection of the pairing.
Now $ay$ must be exact because $(ay)^2 = 0$ and $\Hh^m(P) = \K[h]$ with $\Or_P(h^2) = 1$.
Therefore $\Or_P(ahy)=\Or_{P*}([h][ay])=0$.
By perfection of the pairing, $ah\in a\,\Dd(P^{m-1})$, as required.

For odd $m\ge3$, every element of degree $m$ squares to zero, so we
replace $h$ and $c$ by pairs $h_1,h_2$ and $c_1,c_2$, respectively,
with $[h_1][h_2]=[w]$. Let
\[
 \begin{aligned}
 V&\coloneqq
 \Lambda_\K(a_1,b_{2m-1},h_{1,m},h_{2,m},c_{1,m},c_{2,m})/I,\\
 \Dd c_j&=ah_j,\qquad \Dd a=\Dd b=\Dd h_j=0
 \quad (j=1,2),
 \end{aligned}
\]
where $I$ is generated by
\[
 ac_j,\quad h_1h_2-ab,\quad c_1c_2,\quad
 h_jc_k-\delta_{jk}ab\qquad(1\le j,k\le2),
\]
and all elements of degree $>2m$. Setting
$u_j\coloneqq ah_j=\Dd c_j$, $w\coloneqq ab$, and $\Or(w)=1$, we
have
\[
 V=\Span_\K\{1,a,b,h_1,h_2,c_1,c_2,u_1,u_2,w\},
\]
and
\[
 \Hh(V)=\Span_\K\{[1],[a],[b],[h_1],[h_2],[w]\}.
\]
The products $[a][b]=[h_1][h_2]=[w]$ give Poincar\'e duality on
cohomology.

Form $P=\QQ(\wt V)$ as above, keeping the same symbols for the
images in $P$. Since $\Hh^m(P)=\Span_\K\{[h_1],[h_2]\}$ has a
perfect pairing, it cannot have a basis whose products all vanish.
We will show that $c_j$ can be deformed into closed elements
$z_j\coloneqq c_j+ae_j$, where $e_j\in P^{m-1}$, whose classes
form such a basis. Indeed, $z_jz_k=c_jc_k=0$ by $a^2=ac_j=0$.
We have $\Dd z_j=ah_j-a\,\Dd e_j$, so such $e_j$ exist provided that
$ah_j\in a\,\Dd(P^{m-1})$.
For any $y\in P^{m-1}$ perpendicular to $a\,\Dd(P^{m-1})$, the same
calculation as above gives $\Dd(ay)=0$.
To show that $ay$ is exact, write
$[ay]=\lambda_1[h_1]+\lambda_2[h_2]$ and choose
$x\in P^{m-1}$ with $\Dd x=\lambda_1h_1+\lambda_2h_2-ay$.
Since $x$ has even degree, $\Dd(ax^2)=-2a\,\Dd x\,x$, so
\[
 0=\Or_P(a\,\Dd x\,x)
  =\sum_j\lambda_j\Or_P(\Dd c_j\,x)
  =\sum_j\lambda_j\Or_P(c_j\,\Dd x)
  =-(\lambda_1^2+\lambda_2^2),
\]
using $\Or_P\circ\Dd=0$, $\operatorname{char}\K\ne2$,
$a^2=ac_j=0$, $\Dd c_j=ah_j$, and $c_jh_k=-\delta_{jk}w$.
Since $-1$ is not a square in~$\K$,
we get $\lambda_1=\lambda_2=0$. Thus $ay$ is exact, and the
pairing $\Or_P(ah_jy)=-\Or_{P*}([h_j][ay])$ vanishes.
By perfection of the pairing, $ah_j\in a\,\Dd(P^{m-1})$,
so the required closed elements $z_j=c_j+ae_j$ exist.

It remains to check that $[z_1],[z_2]$ form a basis.
Since $ah_j=a\,\Dd e_j$ and $h_jc_k=\delta_{jk}w$, we have
\[
 \Or_P(h_jz_k)=\delta_{jk}-T_{jk},\qquad
 T_{jk}\coloneqq\Or_P(a\,\Dd e_j\,e_k).
\]
The elements $e_j$ have even degree, so
$T_{jk}+T_{kj}=-\Or_P\bigl(\Dd(ae_je_k)\bigr)=0$.
Thus the pairing matrix is
\[
 \bigl(\Or_P(h_jz_k)\bigr)_{j,k}
 =\begin{pmatrix}1&-T_{12}\\T_{12}&1\end{pmatrix},
\]
with determinant $1+T_{12}^2\ne0$. This proves that $[z_1],[z_2]$
form a basis, whereas $z_jz_k=0$ for all $j,k$, giving the contradiction.
\end{example}

\begin{example}[A minimal Sullivan PDGA not of Hodge type]
\label{Ex:MinimalSullivan}
Let $\operatorname{char}\K=0$ and let $A$ be the minimal Sullivan algebra
\[
 \begin{gathered}
 A\coloneqq\Lambda_\K(a_5,a_9,a_{15},a_{17},a_{21},a_{23}),\\
 \Dd a_{21}=a_5a_{17},\qquad \Dd a_{23}=a_9a_{15},
 \end{gathered}
\]
with the other generators closed. The element
\[
w\coloneqq a_5a_9a_{21}a_{23}
\]
is closed because $a_5^2=a_9^2=0$, and it is not exact because
$A^{57}=0$. Define $\Or\colon A\to\K$ by $\Or(w)=1$ and $\Or=0$ on
all other monomials. Then $\Or\circ\Dd=0$, again because $A^{57}=0$,
and the induced orientation on cohomology satisfies $\Or_*([w])=1$.

We have
\[
 \begin{gathered}
 A^{28}=\K a_5a_{23},\qquad A^{30}=\K a_9a_{21},\\
 \Dd(a_5a_{23})=-a_5a_9a_{15}\ne0,\qquad
 \Dd(a_9a_{21})=a_5a_9a_{17}\ne0,\qquad
 (a_5a_{23})(a_9a_{21})=w.
 \end{gathered}
\]
Therefore, a Hodge decomposition $A=\HH\oplus\im\Dd\oplus C$ would have $C^{28}=A^{28}$
and $C^{30}=A^{30}$, and $C\perp C$ would imply
$\la a_5a_{23},a_9a_{21}\ra=\Or(w)=0$. This is a contradiction, so $A$ admits no Hodge decomposition. 
The same contradiction holds for every chain-level lift of $\Or_*$,
since $w$ is closed and hence every such lift takes the value $1$ on~$w$.

However, the cohomology of $A$ does not satisfy Poincar\'e duality.
For instance, $A^{58}$ is spanned by the closed monomials $w$ and
$a_5a_{15}a_{17}a_{21}$, and $A^{57}=0$, so
\[
 \Hh^{58}(A)=\K[w]\oplus\K[a_5a_{15}a_{17}a_{21}],\qquad
 \Or_*([a_5a_{15}a_{17}a_{21}])=0.
\]
Poincar\'e duality on cohomology means that the musical isomorphisms
\[
 \flat_k\colon\Hh^k\longrightarrow(\Hh^{58-k})^*,\qquad
 \flat_k([z])\coloneqq\la[z],-\ra,
\]
are bijective for every $0\le k\le58$. Since the
cohomology is of finite type and the pairing is graded symmetric, we
have $\flat_{58-k}=(-1)^k\flat_k^*$, so $\flat_k$ is bijective if
and only if $\flat_{58-k}$ is. Moreover, $\Hh^j(A)=0$ for
$27\le j\le31$, because $A^{27}=0$, $\Dd A^{28}=A^{29}$, and $\Dd A^{30}=A^{31}$.
It therefore suffices to make $\flat_k$
bijective for $32\le k\le 58$ without changing the cohomology in
degrees $27,\ldots,31$. We do this in stages $k=32,\ldots,58$,
adjoining at stage~$k$ generators $b_{k-1,i}$ and $c_{k,j}$ of
degrees $k-1$ and~$k$, respectively. This leaves degrees below~$31$ unchanged, and
in particular the Hodge obstruction in degrees $28$ and~$30$.

At the first stage $k=32$ we have
\[
 \Hh^{26}(A)=\K[a_9a_{17}]\oplus\K[a_5a_{21}],\qquad
 \Hh^{32}(A)=\K[a_9a_{23}]\oplus\K[a_{15}a_{17}].
\]
The only nonzero pairing between these bases is
$\la[a_9a_{23}],[a_5a_{21}]\ra=-1$, so $\flat_{32}$ has kernel
$\K[a_{15}a_{17}]$ and cokernel represented by the functional dual
to $[a_9a_{17}]$. We adjoin
$b_{31}\coloneqq b_{31,1}$ with $\Dd b_{31}=a_{15}a_{17}$ to kill the
kernel, and a closed $c_{32}\coloneqq c_{32,1}$ to fill the cokernel,
extending $\Or$ by $\Or(c_{32}a_9a_{17})=1$ and
$\Or(c_{32}a_5a_{21})=0$. No extension of $\Or$ is needed for
$b_{31}$ because $b_{31}A^{27}=0$. This also accidentally kills the second
class in degree~$58$, because
$\Dd(a_5b_{31}a_{21})=-a_5a_{15}a_{17}a_{21}$.

In general, let $A^{(k-1)}$ be the algebra after stage $k-1$, where
$A^{(31)}\coloneqq A$, and let~$\flat_k$ be its musical isomorphism in
degree~$k$. We choose cocycles $z_i$ representing a basis of
$\ker\flat_k$ and adjoin $b_{k-1,i}$ with $\Dd b_{k-1,i}=z_i$. We
extend $\Or$ to the monomials $b_{k-1,i}A^{59-k}$ such that
\[
 \Or(b_{k-1,i}\Dd x)=(-1)^k\Or(z_ix)\qquad(x\in A^{58-k}).
\]
This is possible because $[z_i]\in\ker\flat_k$ makes the
right-hand side vanish for closed~$x$. This condition is necessary for
$\Or(\Dd(b_{k-1,i}x))=0$. We then choose functionals $f_j$
representing a basis of the cokernel of $\flat_k$ and adjoin closed
$c_{k,j}$. We extend $\Or$ to the monomials $c_{k,j}A^{58-k}$ by
\[
\Or(c_{k,j}x)\coloneqq f_j([x])
\]
for closed~$x$ and by zero on a
complement of the cocycles in~$A^{58-k}$. This vanishes on
boundaries, so $\Or(\Dd(c_{k,j}y))=0$ for $y\in A^{57-k}$. Every new
monomial of degree at most~$58$ contains exactly one new generator,
because two would have degree at least $2(k-1)\ge62$. Hence these
prescriptions extend $\Or$ to $A^{(k)}$ with $\Or\circ\Dd=0$, leaving
it unchanged on~$A^{(k-1)}$.

Since $A^{(k-1)}$ has no elements of degrees $1$ to~$4$, the new
elements of degree at most~$k$ are the linear combinations of the new
generators. Hence the cohomology in degrees below $k-1$ is unchanged,
and so is the cohomology in degree $k-1$: if
$\sum_i\alpha_ib_{k-1,i}+y$ with $y\in A^{(k-1)}$ is closed, then
$\sum_i\alpha_i[z_i]=0$, so all $\alpha_i$ vanish. Together with the
unchanged orientation on $A^{(k-1)}$, this preserves the previously
established bijectivity of $\flat_{k^\prime}$ for $32\le k^\prime<k$. In degree~$k$,
the classes $[z_i]$ become exact, the new cocycles are the linear
combinations of the $c_{k,j}$, and $\flat_k([c_{k,j}])=f_j$, so
$\flat_k$ becomes bijective;
here $\Hh^{58-k}$ is unchanged because $58-k<k-1$. Each stage adds
finitely many generators because $A^{(k-1)}$ is of finite type, and
their differentials are decomposable because $A^{(k-1)}$ has no
generator of degree~$k$. After stage~$58$ we therefore have a minimal
Sullivan algebra of finite type whose cohomology satisfies
Poincar\'e duality in degrees $0,\ldots,58$.

For $k>58$, we continue adjoining generators of degree $k-1$ whose
differentials represent a basis of~$\Hh^k$, extending $\Or$ by zero
on the new monomials. No new element has degree~$57$, so
$\Or\circ\Dd=0$ persists, and the cohomology in degrees below~$k$ is
unchanged as before. We obtain a possibly infinitely generated
oriented minimal Sullivan algebra of finite type which is
$4$-connected and whose cohomology is a Poincar\'e duality algebra of
degree~$58$ with $\Or_*([w])=1$. Its degrees $28$ and $30$ are those
of~$A$, so it is not of Hodge type by the obstruction above.

Finally, over $\mathbb Q$, since $n=58$ is bigger than~$4$ and not
divisible by~$4$, Sullivan's rational surgery
theorem~\cite[Theorem~13.2]{Sullivan1977} realizes this algebra as
the minimal model of a simply connected closed smooth $58$-manifold.
The orientation induced by the fundamental class is a nonzero scalar
multiple of~$\Or_*$, so the same obstruction applies.
\end{example}

\begin{samepage}
\renewcommand*{\bibfont}{\footnotesize}
\printbibliography
\end{samepage}
\end{document}